\documentclass[preprint]{elsarticle}

\usepackage{amsmath, amsfonts, amssymb, amsthm}
\usepackage{hyperref}
\usepackage{graphics}
\usepackage{textcomp}
\usepackage[usenames]{color}

{\theoremstyle{plain}
\newtheorem{theorem}{Theorem}[section]
\newtheorem{conjecture}[theorem]{Conjecture}
\newtheorem{proposition}[theorem]{Proposition}
\newtheorem{lemma}[theorem]{Lemma}
\newtheorem{corollary}[theorem]{Corollary}
}
{\theoremstyle{definition}
\newtheorem{definition}[theorem]{Definition}

\newtheorem{example}[theorem]{Example}
}

\numberwithin{equation}{section}
 
\newcommand{\la}{\lambda}
\newcommand{\lp}{{\la^+}}
\newcommand{\mm}{{\mu^-}}
\newcommand{\si}{\sigma}

\newcommand{\hcirc}[1]{\!\!\raisebox{0.5pt}{\textcircled{\raisebox{-1pt}{#1}}}}
\newcommand{\ucirc}[1]{\!\!\raisebox{0.5pt}{\textcircled{\raisebox{-1pt}{#1}}}}

\newcommand{\SSYT}{SSYT}
\newcommand{\SSYTs}{SSYTs}
\newcommand{\extern}{\leftarrow}
\newcommand{\intern}{\leftarrow}
\newcommand{\rowinsert}{\leftarrow}
\newcommand{\D}{\mathrm{D}}
\newcommand{\U}{\mathrm{U}}
\newcommand{\isgood}{exits right}
\newcommand{\begood}{exit right}

\newcommand{\assyt}{anti-semistandard Young tableau}
\newcommand{\ASSYT}{ASSYT}
\newcommand{\ASSYTs}{ASSYTs}

\journal{Journal of Combinatorial Theory, Series A}
\begin{document}

\begin{frontmatter}

\title{A Pieri rule for skew shapes}

\author{Sami H. Assaf\fnref{fn1}}
\address{Department of Mathematics, Massachusetts Institute of Technology, Cambridge, MA 02139, USA}
\ead{sassaf@math.mit.edu}
\fntext[fn1]{Supported by NSF Postdoctoral Fellowship DMS-0703567}

\author{Peter R. W. McNamara}
\address{Department of Mathematics, Bucknell University, Lewisburg, PA 17837, USA}
\ead{peter.mcnamara@bucknell.edu}

\begin{keyword}
  Pieri rule \sep skew Schur functions \sep  Robinson-Schensted
  \MSC 05E05 \sep 05E10 \sep 20C30
\end{keyword}

\begin{abstract} 
  The Pieri rule expresses the product of a Schur function and a
  single row Schur function in terms of Schur functions.  We extend
  the classical Pieri rule by expressing the product of a skew Schur
  function and a single row Schur function in terms of skew Schur
  functions.  Like the classical rule, our rule involves simple
  additions of boxes to the original skew shape.
  Our proof is purely combinatorial and extends
the combinatorial proof of the classical case.  
\end{abstract}

\end{frontmatter}

\section{Introduction}\label{sec:intro} 

The basis of Schur functions is arguably the most interesting and
important basis for the ring of symmetric functions.  This is due not
just to their elegant combinatorial definition, but more broadly to
their connections to other areas of mathematics.  For example, they
are intimately tied to the cohomology ring of the Grassmannian, and
they appear in the representation theory of the symmetric group and of
the general and special linear groups.

It is therefore natural to consider the expansion of the product
$s_\la s_\mu$ of two Schur functions in the basis of Schur functions.
The Littlewood--Richardson rule \cite{LiRi34,Sch77,ThoThesis,Tho78},
which now comes in many different forms (\cite{ECII} is one starting
point), allows us to determine this expansion.  However, more basic
than the Littlewood--Richardson rule is the Pieri rule, which gives a
simple, beautiful and more intuitive answer for the special case when
$\mu = (n)$, a partition of length 1.  Though we will postpone the
preliminary definitions to Section~\ref{sec:prelims} and the statement
of the Pieri rule to Section~\ref{sec:skewpieri}, stating the rule in
a rough form will give its flavor.  For a partition $\la$ and a
positive integer $n$, the Pieri rule states that $s_\la s_n$ is a sum
of Schur functions $s_{\lp}$, where $\lp$ is obtainable by adding
cells to the diagram of $\la$ according to a certain simple rule.  The
Pieri rule's prevalence is highlighted by its adaptions to many other
settings, including Schubert polynomials
\cite{LaSc82,LeSo07,Man98,Sot96,Win98}, LLT polynomials \cite{Lam05},
Hall--Littlewood polynomials \cite{Mor64}, Jack polynomials
\cite{Las89,Sta89}, and Macdonald polynomials
\cite{Koo88,Mac95}.

It is therefore surprising that there does not appear to be a known
adaption of the Pieri rule to the most well-known generalization of
Schur functions, namely skew Schur functions.  We fill this gap in the
literature with a natural extension of the Pieri rule to the skew
setting.  Reflecting the simplicity of the classical Pieri rule, the
skew Pieri rule states that for a skew shape $\la/\mu$ and a positive
integer $n$, $s_{\la/\mu} s_n$ is a signed sum of skew Schur functions
$s_{\lp/\mm}$, where $\lp/\mm$ is obtainable by adding cells to the
diagram of $\la/\mu$ according to a certain simple rule.  Our proof is
purely combinatorial, using a sign-reversing involution that reflects
the combinatorial proof of the classical Pieri rule.  After reading an earlier version of this manuscript, which included an algebraic proof of the case $n=1$ due to Richard Stanley, Thomas Lam provided a complete algebraic proof of our skew Pieri rule.

It is natural to ask if our skew Pieri rule can be extended to give a ``skew'' version of the Littlewood--Richardson rule, and we include such a rule as a conjecture in Section~\ref{sec:conclusion}.  This conjecture has been proved by Lam, Aaron Lauve and Frank Sottile in \cite{LLS09pr} using Hopf algebras.  It remains an open problem to find a combinatorial proof of the skew Littlewood--Richardson rule.

The remainder of this paper is organized as follows.  In
Section~\ref{sec:prelims}, we give the necessary symmetric function
background.  In Section~\ref{sec:skewpieri}, we state the classical
Pieri rule and introduce our skew Pieri rule.  In
Section~\ref{sec:insertion}, we give a variation from \cite{SaSt90} of
the Robinson--Schensted--Knuth algorithm, along with relevant
properties.  This algorithm is then used in
Section~\ref{sec:involution} to define our sign-reversing involution,
which we then use to prove the skew Pieri rule.  We conclude in
Section~\ref{sec:conclusion} with two connections to the Littlewood--Richardson rule.
Lam's algebraic proof appears in Appendix~\ref{sec:algebraic}.

\section{Preliminaries}\label{sec:prelims}

We follow the terminology and notation of \cite{Mac95,ECII} for
partitions and tableaux, except where specified.  Letting $\mathbb{N}$
denote the nonnegative integers, a \emph{partition} $\la$ of $n \in \mathbb{N}$ is a
weakly decreasing sequence $(\la_1, \la_2, \ldots \la_l)$ of positive
integers whose sum is $n$.  It will be convenient to set $\la_k = 0$
for $k > l$.  We also let $\emptyset$ denote the unique partition with
$l=0$.  We will identify $\la$ with its \emph{Young diagram} in
``French notation'': represent the partition $\la$ by the unit square
cells with top-right corners $(i,j) \in \mathbb{N} \times\mathbb{N}$
such that $1 \leq i \leq \la_j$.  For example, the partition
$(4,2,1)$, which we abbreviate as 421, has Young diagram
\setlength{\unitlength}{4mm}
\[
\begin{picture}(4,3)(0,0)
\put(0,0){\line(1,0){4}}
\put(0,1){\line(1,0){4}}
\put(0,2){\line(1,0){2}}
\put(0,3){\line(1,0){1}}
\multiput(0,3)(1,0){2}{\line(0,-1){3}}
\put(2,2){\line(0,-1){2}}
\multiput(3,1)(1,0){2}{\line(0,-1){1}}
\put(4.5,1.5){.}
\end{picture}
\]
Define the \emph{conjugate} or \emph{transpose} $\la^t$ of $\la$ to be
the partition with $\la_i$ cells in column $i$.  For example, $421^t =
3211$.  For another partition $\mu$, we write $\mu \subseteq \la$
whenever $\mu$ is contained within $\la$ (as Young diagrams);
equivalently $\mu_i \leq \la_i$ for all $i$. In this case, we define
the {\em skew shape} $\la / \mu$ to be the set theoretic difference
$\la - \mu$.  In particular, the partition $\la$ is the skew shape
$\la/\emptyset$.  We call the number of cells of $\la/\mu$ its
\emph{size}, denoted $|\la/\mu|$.  We say that a skew shape forms a
\emph{horizontal strip} (resp.\ \emph{vertical strip}) if it
contains no two cells in the same column (resp.\ row).  A
\emph{$k$-horizontal strip} is a horizontal strip of size $k$, and
similarly for vertical strips.  For example, the skew
shape $421/21$ is a 4-horizontal strip: \setlength{\unitlength}{4mm}
\[
\begin{picture}(4,3)(0,0)
\put(2,0){\line(1,0){2}}
\put(1,1){\line(1,0){3}}
\put(0,2){\line(1,0){2}}
\put(0,3){\line(1,0){1}}
\put(0,3){\line(0,-1){1}}
\put(1,3){\line(0,-1){2}}
\put(2,2){\line(0,-1){2}}
\multiput(3,1)(1,0){2}{\line(0,-1){1}}
\put(4.5,1.5){.}
\end{picture}
\]
With another skew shape $\sigma/\tau$, we let $(\la/\mu) *
(\sigma/\tau)$ denote the skew shape obtained by positioning $\la/\mu$
so that its bottom right cell is immediately above and left of the top
left cell of $\sigma/\tau$.  For example, the horizontal strip 421/21
above could alternatively be written as $(21/1) * (2)$ or as $(1) *
(31/1)$.

A {\em Young tableau} of shape $\la/\mu$ is a map from the cells of $\la/\mu$ to the positive integers.
A {\em semistandard Young tableau} (SSYT) is
such a filling which is weakly increasing from left-to-right along
each row and strictly increasing up each column, such as
\setlength{\unitlength}{4mm}
\[
\begin{picture}(4,3)(0,0)
\put(1,0){\line(1,0){3}}
\put(0,1){\line(1,0){4}}
\put(0,2){\line(1,0){3}}
\put(0,3){\line(1,0){1}}
\put(0,3){\line(0,-1){2}}
\put(1,3){\line(0,-1){3}}
\multiput(2,2)(1,0){2}{\line(0,-1){2}}
\put(4,1){\line(0,-1){1}}
\put(1.3,0.2){1}
\put(2.3,0.2){2}
\put(3.3,0.2){7}
\put(0.3,1.2){3}
\put(1.3,1.2){3}
\put(2.3,1.2){5}
\put(0.3,2.2){5}
\put(4.5,1.5){.}
\end{picture}
\]
The \emph{content} of an SSYT $T$ is the sequence $\pi$ such
that $T$ has $\pi_i$ cells with entry $i$; in this case $\pi =
(1,1,2,0,2,0,1)$.

We let $\Lambda$ denote the ring of symmetric functions in the
variables $x = (x_1, x_2, \ldots)$ over $\mathbb{Q}$, say.  We will
use three familiar bases from \cite{Mac95,ECII} for $\Lambda$: the
elementary symmetric functions $e_{\la}$, the complete homogeneous
symmetric functions $h_{\la}$ and, most importantly, the \emph{Schur
  functions} $s_{\la}$. The Schur functions form an orthonormal basis
for $\Lambda$ with respect to the Hall inner product and may be
defined in terms of \SSYTs\ by
\begin{equation}\label{equ:schur}
  s_{\la} = \sum_{T \in \mathrm{\SSYT}(\la)} x^{T} ,
\end{equation}
where the sum is over all \SSYTs\ of shape $\la$ and where $x^T$
denotes the monomial $x_{1}^{\pi_1} x_{2}^{\pi_2} \cdots$ when $T$ has
content $\pi$.  Replacing $\la$ by $\la/\mu$ in \eqref{equ:schur}
gives the definition of the {\em skew Schur function} $s_{\la/\mu}$,
where the sum is now over all \SSYTs\ of shape $\la/\mu$.  For
example, the \SSYT\ shown above contributes the monomial $x_1 x_2
x_3^2 x_5^2 x_7$ to $s_{431/1}$.

\section{The skew Pieri rule}\label{sec:skewpieri}

The celebrated Pieri rule gives an elegant method for expanding the
product $s_{\la} s_n$ in the Schur basis.  This rule was originally
stated in \cite{Pie1893} in the setting of Schubert Calculus. Recall
that the single row Schur function $s_n$ equals the complete
homogeneous symmetric function $h_n$. Recall also the involution
$\omega$ on $\Lambda$, which may be defined by sending the Schur
function $s_{\la}$ to $s_{\la^t}$ or equivalently by sending $h_k$ to
$e_k$. Thus the Schur function $s_{1^n}$ equals the elementary
symmetric function $e_{n}$, where $1^n$ denotes a single column of
size $n$.

\begin{theorem}[\cite{Pie1893}]\label{thm:pieri} 
  For any partition $\la$ and positive integer $n$, we have
  \begin{equation}\label{equ:pieri}
    s_\la s_n = s_{\la} h_n = \sum_
    {\lp/\la \ n\text{-}\mathrm{hor.\ strip}}
    s_{\lp}, 
        \end{equation}
        where the sum is over all partitions $\lp$ such that $\lp/\la$
        is a horizontal strip of size $n$.
\end{theorem}

Applying the involution $\omega$ to \eqref{equ:pieri}, we get the dual
version of the Pieri rule:
\begin{equation}\label{equ:dualpieri}
  s_\la s_{1^n} = s_\la e_n = \sum_{\lp/\la \ n\text{-}\mathrm{vert.\ strip}}
  s_\lp\ ,
\end{equation}
where the sum is now over all partitions $\lp$ such that $\lp/\la$ is
a vertical strip of size $n$.  

A simple application of Theorem~\ref{thm:pieri} gives
\[
s_{322} s_2 = s_{3222} + s_{3321} + s_{4221} + s_{432} + s_{522},
\]
as represented diagrammatically in Figure~\ref{fig:pieriexample}.

\begin{figure}[htbp]
  \begin{center}
    \[
    \scalebox{.7}{\includegraphics{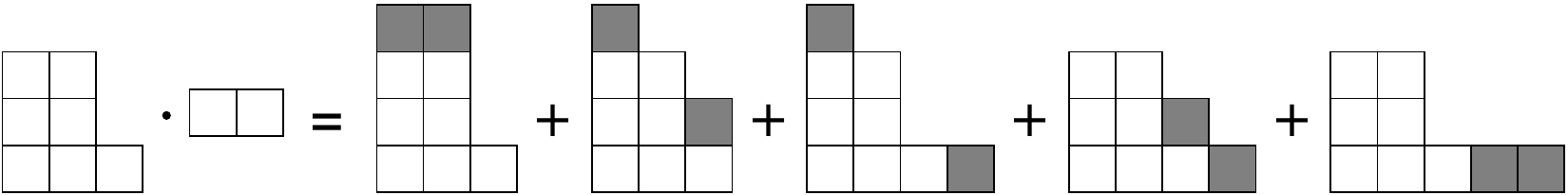}}
    \]
    \caption{The expansion of $s_{322} s_2$ by the Pieri rule.}
    \label{fig:pieriexample}
  \end{center}
\end{figure}

Given the simplicity of \eqref{equ:pieri}, it is natural to hope for a
simple expression for $s_{\la/\mu} s_n$ in terms of skew Schur
functions. This brings us to our main result.

\begin{theorem}\label{thm:skewpieri}
  For any skew shape $\la/\mu$ and positive integer $n$, we have
  \begin{equation}\label{equ:skewpieri}
    s_{\la/\mu} s_n = s_{\la/\mu} h_n  = \sum_{k=0}^n (-1)^k
    \sum_{\substack{\lp/\la \ (n-k)\text{-}\mathrm{hor.\ strip} \\
        \mu/\mm \ k \text{-} \mathrm{vert.\ strip}}} s_{\lp/\mm}\ ,  
  \end{equation}
  where the second sum is over all partitions $\lp$ and $\mm$ such that
  $\lp/\la$ is a horizontal strip of size $n-k$ and $\mu/\mm$ is a
  vertical strip of size $k$.
\end{theorem}

Observe that when $\mu=\emptyset$, Theorem~\ref{thm:skewpieri}
specializes to Theorem~\ref{thm:pieri}.  Again, we can apply the
$\omega$ transformation to obtain the dual version of the skew Pieri
rule.

\begin{corollary}  
  For any skew shape $\la/\mu$ and any positive integer $k$, we have
  \[
  s_{\la/\mu} s_{1^n} = s_{\la/\mu} e_n = \sum_{k=0}^n (-1)^k
  \sum_{\substack{\lp/\la \ (n-k)\text{-} \mathrm{vert.\ strip} \\
        \mu/\mm \ k \text{-} \mathrm{hor.\ strip}}} s_{\lp/\mm}, 
  \]
  where the sum is over all partitions $\lp$ and $\mm$ such that
  $\lp/\la$ is a vertical strip of size $n-k$ and $\mu/\mm$ is a
  horizontal strip of size $k$.
\end{corollary}

\begin{example}\label{exa:skewpieri}
  A direct application of Theorem~\ref{thm:skewpieri} gives
  \begin{align*}
    s_{322/11} s_2 &= 
    s_{3222/11} + s_{3321/11} + s_{4221/11} +
   s_{432/11} + s_{522/11} \\ 
       & \qquad - s_{3221/1} - s_{332/1} -s_{422/1} + s_{322},
  \end{align*}
  as represented diagrammatically in
  Figure~\ref{fig:skewpieriexample}. 
\end{example}

\begin{figure}[htbp]
  \begin{center}
    \[
    \scalebox{.7}{\includegraphics{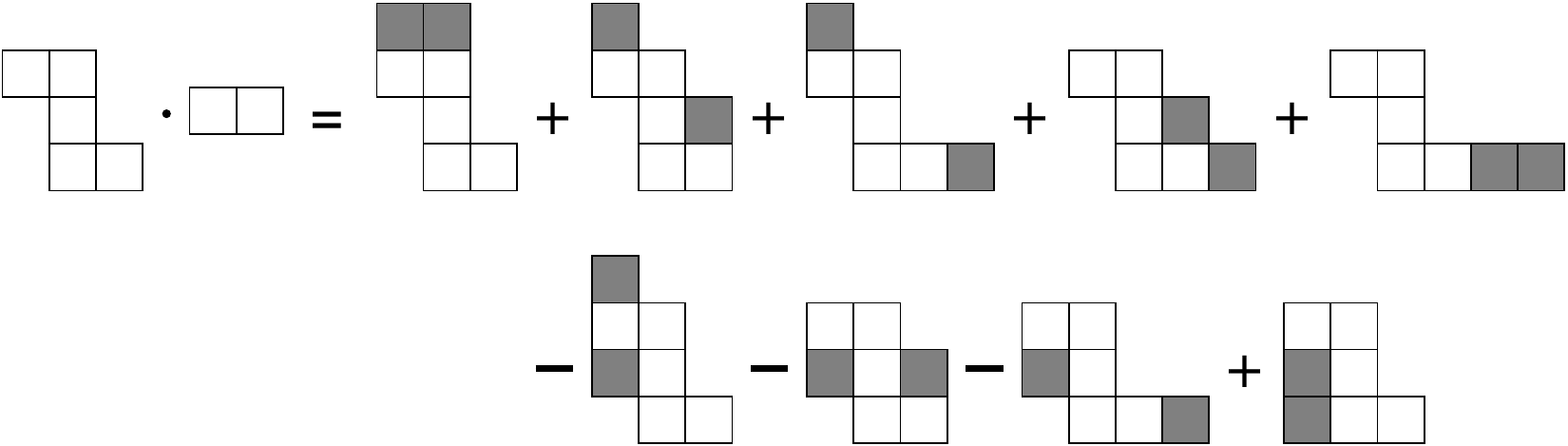}}
    \]
    \caption{The expansion of $s_{322/11} s_2$
    by the skew Pieri rule.}
    \label{fig:skewpieriexample}
  \end{center}
\end{figure}

As \eqref{equ:skewpieri} contains negative signs, our approach to
proving Theorem~\ref{thm:skewpieri} will be to construct a
sign-reversing involution on \SSYTs\ of shapes of the form $\lp/\mm$.
We will then provide a bijection between \SSYTs\ of shape $\lp/\mm$
that are fixed under the involution and \SSYTs\ of shape $(\la/\mu) *
(n)$.  The result then follows from the fact that $s_{(\la/\mu)*(n)} = s_{\la/\mu} s_n$.

\section{Row insertion}\label{sec:insertion}

In order to describe our sign-reversing involution, we will need the
Robinson--Schensted--Knuth (RSK) row insertion algorithm on \SSYTs\
\cite{Rob1938,Sch1961,Knu70}. For a thorough treatment of this
algorithm along with many applications, we recommend \cite{ECII}.  In
fact, we will use an analogue of the algorithm for \SSYTs\ of
\emph{skew} shape from \cite{SaSt90}.  There, row insertion comes in
two forms, external and internal row insertion.  External row
insertion, which we now define, is just like the classical RSK
insertion.

\begin{definition}\label{def:extern}
  Let $T$ be an \SSYT\ of arbitrary skew shape and choose a positive
  integer $k$. Define the \emph{external row insertion of $k$ into
    $T$}, denoted $T \extern_0 k$, as follows: if $k$ is weakly larger
  than all entries in row $1$ of $T$, then add $k$ to the right end of
 the row and terminate the process. Otherwise, find the leftmost
  cell in row $1$ of $T$ whose entry, say $k'$, is greater than
  $k$. Replace this entry by $k$ and then row insert $k'$ into $T$ at
  row $2$ using the procedure just described.  Repeat the process
  until some entry comes to rest at the right end of a row.
\end{definition}

\begin{example}\label{exa:extern}
  Let $\la/\mu=7541/32$ and $\lp/\mm=7542/31$ so that the outlined
  entries below are those in $\la/\mu$.  The result of externally row
  inserting a 2 is shown below, with changed cells circled.
\begin{equation}\label{equ:extern}
\setlength{\unitlength}{4.5mm}
\begin{picture}(7,1.5)(0,2.5)
\put(0,2){\line(0,1){2}}
\put(1,3){\line(0,1){1}}
\put(2,1){\line(0,1){1}}
\put(3,0){\line(0,1){1}}
\put(4,2){\line(0,1){1}}
\put(5,1){\line(0,1){1}}
\put(7,0){\line(0,1){1}}
\put(0,4){\line(1,0){1}}
\put(0,2){\line(1,0){2}}
\put(1,3){\line(1,0){3}}
\put(2,1){\line(1,0){1}}
\put(3,0){\line(1,0){4}}
\put(4,2){\line(1,0){1}}
\put(5,1){\line(1,0){2}}
\put(0.35,3.25){4}
\put(1.35,3.25){5}
\put(2.35,3.25){\hcirc{\ }}
\put(0.35,2.25){2}
\put(1.35,2.25){2}
\put(2.35,2.25){\hcirc{7}}
\put(3.35,2.25){7}
\put(1.35,1.25){1}
\put(2.35,1.25){2}
\put(3.35,1.25){3}
\put(4.35,1.25){\hcirc{4}}
\put(3.35,0.25){2}
\put(4.35,0.25){2}
\put(5.35,0.25){\hcirc{3}}
\put(6.35,0.25){6}
\end{picture}
\extern_0 2 \qquad = \qquad
\begin{picture}(7,1.5)(0,2.5)
\put(0,2){\line(0,1){2}}
\put(1,3){\line(0,1){1}}
\put(2,1){\line(0,1){1}}
\put(3,0){\line(0,1){1}}
\put(4,2){\line(0,1){1}}
\put(5,1){\line(0,1){1}}
\put(7,0){\line(0,1){1}}
\put(0,4){\line(1,0){1}}
\put(0,2){\line(1,0){2}}
\put(1,3){\line(1,0){3}}
\put(2,1){\line(1,0){1}}
\put(3,0){\line(1,0){4}}
\put(4,2){\line(1,0){1}}
\put(5,1){\line(1,0){2}}
\put(0.35,3.25){4}
\put(1.35,3.25){5}
\put(2.35,3.25){\ucirc{7}}
\put(0.35,2.25){2}
\put(1.35,2.25){2}
\put(2.35,2.25){\ucirc{4}}
\put(3.35,2.25){7}
\put(1.35,1.25){1}
\put(2.35,1.25){2}
\put(3.35,1.25){3}
\put(4.35,1.25){\ucirc{3}}
\put(3.35,0.25){2}
\put(4.35,0.25){2}
\put(5.35,0.25){\ucirc{2}}
\put(6.35,0.25){6}
\end{picture} 
\end{equation}
\vspace{1.5\unitlength}
\end{example}

An \emph{inside corner} (resp.\ \emph{outside corner}) of an SSYT $T$
is a cell that has no cell of $T$ immediately below or to its left
(resp.\ above or to its right).  Therefore, inside and outside corners
are those individual cells whose removal from $T$ still yields an SSYT
of skew shape.

\begin{definition}\label{def:intern}
  Let $T$ be an \SSYT\ of arbitrary skew shape and let $T$ have an
  inside corner in row $r$ with entry $k$.  Define the \emph{internal
    row insertion of $k$ from row $r$ into $T$}, denoted $T \intern_r
  k$, as the removal of $k$ from row $r$ and its insertion, using the
  rules for external row insertion, into row $r+1$.  The insertion
  proceeds until some entry comes to rest at the right end of a row.
\end{definition}
 
We could regard external insertions as internal insertions from row $0$,
explaining our notation.  We will simply write $T \rowinsert k$
when specifying the type or row of the insertion
is unnecessarily cumbersome.

\begin{example} Taking $T$ as the SSYT on the right
  in~\eqref{equ:extern}, the result of internally row inserting the 1
  from row 2 is shown below.
\begin{equation}\label{equ:intern}
\setlength{\unitlength}{4.5mm}
\begin{picture}(7,2.5)(0,2.5)
\put(0,2){\line(0,1){2}}
\put(1,3){\line(0,1){1}}
\put(2,1){\line(0,1){1}}
\put(3,0){\line(0,1){1}}
\put(4,2){\line(0,1){1}}
\put(5,1){\line(0,1){1}}
\put(7,0){\line(0,1){1}}
\put(0,4){\line(1,0){1}}
\put(0,2){\line(1,0){2}}
\put(1,3){\line(1,0){3}}
\put(2,1){\line(1,0){1}}
\put(3,0){\line(1,0){4}}
\put(4,2){\line(1,0){1}}
\put(5,1){\line(1,0){2}}
\put(0.35,4.25){\hcirc{\ }}
\put(0.35,3.25){\hcirc{4}}
\put(1.35,3.25){5}
\put(2.35,3.25){7}
\put(0.35,2.25){\hcirc{2}}
\put(1.35,2.25){2}
\put(2.35,2.25){4}
\put(3.35,2.25){7}
\put(1.35,1.25){\hcirc{1}}
\put(2.35,1.25){2}
\put(3.35,1.25){3}
\put(4.35,1.25){3}
\put(3.35,0.25){2}
\put(4.35,0.25){2}
\put(5.35,0.25){2}
\put(6.35,0.25){6}
\end{picture}
\intern_2 1 \qquad = \qquad
\begin{picture}(7,2.5)(0,2.5)
\put(0,2){\line(0,1){2}}
\put(1,3){\line(0,1){1}}
\put(2,1){\line(0,1){1}}
\put(3,0){\line(0,1){1}}
\put(4,2){\line(0,1){1}}
\put(5,1){\line(0,1){1}}
\put(7,0){\line(0,1){1}}
\put(0,4){\line(1,0){1}}
\put(0,2){\line(1,0){2}}
\put(1,3){\line(1,0){3}}
\put(2,1){\line(1,0){1}}
\put(3,0){\line(1,0){4}}
\put(4,2){\line(1,0){1}}
\put(5,1){\line(1,0){2}}
\put(0.35,4.25){\ucirc{4}}
\put(1.35,3.25){5}
\put(2.35,3.25){7}
\put(0.35,3.25){\ucirc{2}}
\put(1.35,2.25){2}
\put(2.35,2.25){4}
\put(3.35,2.25){7}
\put(0.35,2.25){\ucirc{1}}
\put(1.35,1.25){\ucirc{\ }}
\put(2.35,1.25){2}
\put(3.35,1.25){3}
\put(4.35,1.25){3}
\put(3.35,0.25){2}
\put(4.35,0.25){2}
\put(5.35,0.25){2}
\put(6.35,0.25){6}
\end{picture}
\end{equation}
\vspace{2\unitlength}
\end{example}

For both types of insertion, we must be a little careful when
inserting an entry into an empty row, say row $i$: in this case $\la_i
= \mu_i$ and the entry must be placed in column $\la_i+1$.

Note that an internal insertion results in the same multiset of
entries while an external insertion adds an entry. It is not
difficult to check that either operation results in an SSYT.

We will also need to invert row insertions, again for skew
shapes and following~\cite{SaSt90}.

\begin{definition}\label{def:deletion}
  Let $T$ be an \SSYT\ of arbitrary skew shape and choose an outside
  corner $c$ of $T$, say with entry $k$. Define the \emph{reverse row
    insertion of $c$ from $T$}, denoted $T \rightarrow c$, by deleting
  $c$ from $T$ and reverse inserting $k$ into the row below, say row
  $r$, as follows: if $r=0$, then the procedure
  terminates.  Otherwise, if $k$ is weakly smaller than all entries in row $r$,
  then place $k$ at the left end of row $r$ and terminate the
  procedure.  Otherwise, find the rightmost cell in row $r$ whose entry, say $k'$,
  is less than $k$. Replace this entry by $k$ and then reverse row
  insert $k'$ into row $r-1$ using the procedure just described.
\end{definition}

\begin{example}
  In~\eqref{equ:extern}, reverse row insertion of the cell containing
  the circled 7 from the \SSYT\ $T$ on the right results in the
  \SSYT\ on the left, and similarly in~\eqref{equ:intern} for the
  circled 4.
\end{example}

As with row insertion, it follows from the definition that the
resulting array will again be an \SSYT.  Observe that the first type
of termination mentioned in Definition~\ref{def:deletion} corresponds
to reverse external row insertion, and we then say that \emph{$k$
  lands in row $0$}.  The second type of termination corresponds to
reverse internal row insertion, and we then say that \emph{$k$ lands
  in row $r$}.  In both cases, we will call the entry $k$ left at the
end of the procedure the \emph{final entry of $T \rightarrow c$}.  The
following lemma, which follows immediately from
Definitions~\ref{def:extern}, \ref{def:intern} and \ref{def:deletion},
formalizes the bijectivity of row and reverse row insertion.

\begin{lemma}\label{lem:inverse}
  Let $T$ be an \SSYT\ of skew shape.
  \begin{enumerate}
  \renewcommand{\theenumi}{\alph{enumi}}
  \item If $S$ is the result of $T \rowinsert k$ for some positive
    integer $k$, then $S \rightarrow c$ results in $T$, where $c$ is
    the unique non-empty cell of $S$ that is empty in $T$.
  \item If $S$ is the result of $T \rightarrow c$ for some removable
    cell $c$ of $T$ and the final entry $k$ of $T \rightarrow c$ lands
    in row $r \geq 0$, then $S \rowinsert_r k$ results in $T$.
  \end{enumerate}  
\end{lemma}

For both row insertion and reverse row insertion, we will often want
to track the cells affected by the procedure. Therefore define the
\emph{bumping path of the row insertion $T \rowinsert k$} (resp.\ the
\emph{reverse bumping path of a reverse row insertion $T \rightarrow
  c$}) to be the set of cells in $T$, as well as those empty cells, where the
entries differ from the corresponding entries in $T \rowinsert k$
(resp.\ $T \rightarrow c$).  The cells of the bumping paths for row
insertion and reverse row insertion are circled in~\eqref{equ:extern}
and~\eqref{equ:intern}. Note that the fact that the bumping path and
reverse bumping path are equal in each of these examples is a
consequence of Lemma~\ref{lem:inverse}.

It is easy to see that the bumping paths always move weakly right from
top to bottom in the case of column-strict tableaux. The following
bumping lemma will play a crucial role in defining our sign-reversing
involution and in proving its relevant properties.

\begin{lemma}\label{lem:horizontal} 
  Let $T$ be an \SSYT\ of skew shape and let $k,k'$ be positive
  integers. Let $B$ be the bumping path of $T \rowinsert k$ and let
  $B'$ be the bumping path of $(T \rowinsert k) \rowinsert k'$.
  \begin{enumerate}
    \renewcommand{\theenumi}{\alph{enumi}}
  \item If $B$ is strictly left of $B'$ in any row $r$, then $B$ is
    strictly left of $B'$ in every row they both occupy.  Moreover,
    the top cells of $B$ and $B'$ form a horizontal strip.
  \item If both row insertions are external, then $B$ is strictly left
    of $B'$ in every row they both occupy if and only if $k \leq k'$.
  \item 
 Suppose $C'$ is the reverse bumping path of $T \rightarrow c'$ with final entry $k'$ and
    $C$ is the reverse bumping path of $(T \rightarrow c') \rightarrow
    c$ with final entry $k$.  If $c$ is strictly left of $c'$, then $C$ is strictly left of
    $C'$ in every row they both occupy.  
    If, in addition,  both reverse
    row insertions land in row 0, then $k \leq k'$.
  \end{enumerate}
  \end{lemma}

\begin{proof}
  For general $i$, let $B_i$ (resp.\ $B_i'$) be the cell of $B$
  (resp.\ $B'$) in row $i$, say with entry $b_i$ in $T$ (resp.\ $b_i'$ in $T \rowinsert k$).

  (a) Suppose $B_r$ is strictly left of $B_r'$ for some $r$ for which
  $B_{r+1}$ and $B_{r+1}'$ both exist.  Then $b_r$ will occupy the
  cell $B_{r+1}$ in $T \rowinsert k$, and since $b_r' \geq b_r$,
  $B_{r+1}'$ will be strictly right of $B_{r+1}$.  To show that
  $B_{r-1}$ is strictly left of $B_{r-1}'$ assuming both exist, let
  $c$ denote the cell that is empty in $T$ but non-empty in $T
  \rowinsert k$ and let $c'$ denote the cell that is empty in $T
  \rowinsert k$ but not in $(T \rowinsert k) \rowinsert k'$.  By
  Lemma~\ref{lem:inverse}(a), $B$ (resp.\ $B'$) is also the reverse
  bumping path of $(((T \rowinsert k) \rowinsert k') \rightarrow
  c')\rightarrow c$ (resp.\ $((T \rowinsert k) \rowinsert k')
  \rightarrow c'$).  We know that $b_{r-1}$ will occupy the cell $B_r$
  in $(T \rowinsert k) \rowinsert k'$ while $b_{r-1}'$ will occupy the
  cell $B'_r$, implying $b_{r-1} \leq b_{r-1}'$.  As a result,
  considering the reverse row insertions of $c'$ and then $c$, we
  deduce that $B_{r-1}$ is strictly left of $B_{r-1}'$.
  
  Now consider the top cells $B_r$ and $B_s'$ of $B$ and $B'$
  respectively.  Since $B_r$ is the top cell of $B$, we know that it
  is at the right end of row $r$.  When $s \geq r$, we know that
  $B_r'$ is strictly right of $B_r$.  But then $B_r'$ is empty in $T
  \rowinsert k$ and so we must have $s=r$.  We conclude that $s \leq
  r$ and the result follows.
  
  (b) If $k \leq k'$, then $B_1$ is strictly left of $B_1'$.  If $k >
  k'$, then $B_1$ is weakly right of $B_1'$.  The result now follows
  from (a).

  (c) Letting $T'$ denote $(T \rightarrow c') \rightarrow c$,
  Lemma~\ref{lem:inverse} implies that $C$ equals the bumping path $B$
  of $T' \rowinsert k$ while $C'$ equals the bumping path $B'$ of $(T'
  \rowinsert k)\rowinsert k'$.  Applying (a), it suffices to show that
  $C$ is strictly left of $C'$ in some row which they both occupy.
  Consider row $r$, the top row of $C'$.  Since $c$ is strictly left
  of $c'$, either $C$ and $C'$ have no rows in common, in which case
  the result is trivial, or else $C$ has a cell in row $r$. When we
  choose the elements of $C$, we have already performed the reverse
  row insertion $T \rightarrow c'$.  In particular, the cell $c'$ is
  empty.  Therefore, $C$ must be strictly left of $C'$ in row $r$, as
  required.  The second assertion now follows by applying (b) to these
  new $B$ and $B'$.
\end{proof}

To foreshadow the role of Lemma~\ref{lem:horizontal} in the following
section, we give a proof of the classical Pieri rule using this
result.

\begin{proof}[Proof of Theorem~\ref{thm:pieri}] 
  The formula is proved if we can give a bijection between \SSYTs\ of
  shape $\la * (n)$ and \SSYTs\ of shape $\lp$ such that $\lp/\la$ is
  a horizontal strip of size $n$. Let $k_1 \leq \cdots \leq k_n$ be
  entries of $(n)$ from left to right. Repeated applications of (a)
  and (b) of Lemma~\ref{lem:horizontal} ensure that row inserting
  these entries into an \SSYT\ of shape $\la$ will add a horizontal
  strip of size $n$ to $\la$. By Lemma~\ref{lem:inverse}, this
  establishes a bijection where the inverse map is given by reverse
  row inserting the cells of $\lp/\la$ from right to left.
\end{proof}

\section{A sign-reversing involution}\label{sec:involution}

Throughout this section, fix a skew shape $\la/\mu$. We will be
interested in \SSYTs\ of shape $\lp/\mm$, where we always assume that
$\lp/\la$ is a horizontal strip, $\mu/\mm$ is a vertical strip, and
$|\lp/\la| + |\mu/\mm| = n$.  Our goal is to construct a
sign-reversing involution on SSYTs whose shapes take the form
$\lp/\mm$, such that the fixed points are in bijection with SSYTs of
shape $(\la/\mu) * (n)$.

Our involution is reminiscent of the proof of the classical Pieri rule
given in Section~\ref{sec:insertion}. By Lemma~\ref{lem:inverse},
reverse row insertion gives a bijective correspondence provided we
record the final entry and its landing row.  Our strategy, then, is to
reverse row insert the cells of $\lp/\la$ from right to left,
recording the entries as we go. If at some stage we land in row $r
\geq 1$, we will then re-insert all the previous final entries. More
formally, we have the following definition of a \emph{downward slide
  of $T$}.

\begin{definition}\label{def:downward}
  Let $T$ be an \SSYT\ of shape $\lp/\mm$.  Define the \emph{downward
    slide of $T$}, denoted $\D(T)$, as follows: construct $T
  \rightarrow c_1$ where $c_1$ is the rightmost cell of $\lp/\la$, and
  let $k_1$ denote the final entry. If $k_1$ lands in row $0$, then
  continue with $c_2$ the second rightmost cell of $\lp/\la$ and $k_2$
  the final entry of $(T \rightarrow c_1) \rightarrow c_2$.  Continue
  until the first time $k_m$ lands in row $r \geq 1$ and set $m'=m-1$,
  or set $m = m' = |\lp/\la|$ if no such $k_m$ exists. Then $\D(T)$ is
  given by
    \[
    (\cdots(( (\cdots(T \rightarrow c_1) \rightarrow
    c_2\cdots) \rightarrow c_m ) \extern k_{m'} ) \cdots
    ) \extern k_{1}.
    \]
\end{definition}

\begin{example}\label{exa:downward}
  With $T$ shown on the left below, we exhibit the construction of
  $\D(T)$ in two steps.  We find that $m=4$ and the middle \SSYT\
  shows the result of $(((T \rightarrow c_1) \rightarrow c_2)
  \rightarrow c_3) \rightarrow c_4$.  The entries that land in row 0
  are recorded in the dashed box.  Then the \SSYT\ on the right is
  $\D(T)$.  The significance of the circles will be
  explained later.
\begin{equation}\label{equ:downward}
\setlength{\unitlength}{4.5mm}
\begin{picture}(7,1)(0,5)
\put(0,3){\line(0,1){2}}
\put(1,4){\line(0,1){1}}
\put(2,2){\line(0,1){1}}
\put(3,1){\line(0,1){1}}
\put(4,3){\line(0,1){1}}
\put(5,2){\line(0,1){1}}
\put(7,1){\line(0,1){1}}
\put(0,5){\line(1,0){1}}
\put(0,3){\line(1,0){2}}
\put(1,4){\line(1,0){3}}
\put(2,2){\line(1,0){1}}
\put(3,1){\line(1,0){4}}
\put(4,3){\line(1,0){1}}
\put(5,2){\line(1,0){2}}
\put(0.35,5.25){9}
\put(0.35,4.25){3}
\put(1.35,4.25){\ucirc{5}}
\put(2.35,4.25){7}
\put(3.35,4.25){7}
\put(0.35,3.25){2}
\put(1.35,3.25){\ucirc{2}}
\put(2.35,3.25){3}
\put(3.35,3.25){4}
\put(1.35,2.25){\ucirc{1}}
\put(2.35,2.25){2}
\put(3.35,2.25){2}
\put(4.35,2.25){3}
\put(5.35,2.25){6}
\put(2.35,1.25){\ucirc{\ }}
\put(3.35,1.25){1}
\put(4.35,1.25){2}
\put(5.35,1.25){2}
\put(6.35,1.25){5}
\end{picture}
\quad
\begin{picture}(10,1)(0,5)
\put(0,3){\line(0,1){2}}
\put(1,4){\line(0,1){1}}
\put(2,2){\line(0,1){1}}
\put(3,1){\line(0,1){1}}
\put(4,3){\line(0,1){1}}
\put(5,2){\line(0,1){1}}
\put(7,1){\line(0,1){1}}
\put(0,5){\line(1,0){1}}
\put(0,3){\line(1,0){2}}
\put(1,4){\line(1,0){3}}
\put(2,2){\line(1,0){1}}
\put(3,1){\line(1,0){4}}
\put(4,3){\line(1,0){1}}
\put(5,2){\line(1,0){2}}
\put(0.35,5.25){9}
\put(0.35,4.25){3}
\put(0.35,3.25){2}
\put(1.35,3.25){5}
\put(2.35,3.25){7}
\put(3.35,3.25){7}
\put(1.35,2.25){2}
\put(2.35,2.25){2}
\put(3.35,2.25){3}
\put(4.35,2.25){4}
\put(2.35,1.25){1}
\put(3.35,1.25){2}
\put(4.35,1.25){2}
\put(5.35,1.25){3}
\put(6.35,1.25){6}
\put(7,0){\dashbox{0.25}(3,1){}}
\put(7.35,0.25){1}
\put(8.35,0.25){2}
\put(9.35,0.25){5}
\end{picture}
\quad\!
\begin{picture}(7,1)(0,5)
\put(0,3){\line(0,1){2}}
\put(1,4){\line(0,1){1}}
\put(2,2){\line(0,1){1}}
\put(3,1){\line(0,1){1}}
\put(4,3){\line(0,1){1}}
\put(5,2){\line(0,1){1}}
\put(7,1){\line(0,1){1}}
\put(0,5){\line(1,0){1}}
\put(0,3){\line(1,0){2}}
\put(1,4){\line(1,0){3}}
\put(2,2){\line(1,0){1}}
\put(3,1){\line(1,0){4}}
\put(4,3){\line(1,0){1}}
\put(5,2){\line(1,0){2}}
\put(0.35,6.25){${\ }$}
\put(0.35,5.25){${9}$}
\put(0.35,4.25){3}
\put(1.35,4.25){${5}$}
\put(2.35,4.25){7}
\put(0.35,3.25){2}
\put(1.35,3.25){${3}$}
\put(2.35,3.25){4}
\put(3.35,3.25){7}
\put(1.35,2.25){${2}$}
\put(2.35,2.25){2}
\put(3.35,2.25){2}
\put(4.35,2.25){3}
\put(5.35,2.25){6}
\put(2.35,1.25){${1}$}
\put(3.35,1.25){1}
\put(4.35,1.25){2}
\put(5.35,1.25){2}
\put(6.35,1.25){5}
\end{picture}
\vspace{5\unitlength}
\end{equation}
Alternatively, if $T$ is the SSYT shown on the left below, we find
that $m=3$ and that all three final entries land in row 0.  Then
$m'=|\lp/\la|$ and Lemma~\ref{lem:inverse} ensures that $D(T) = T$.
Below in the middle, we have shown $((T \rightarrow c_1) \rightarrow
c_2) \rightarrow c_3$.  The position of the dashed box is intended to
be suggestive: together with the entries in the outlined shape, we see
that we have an SSYT of shape $(\la/\mu) * (n) = (653/21)*(3)$.
\[
\setlength{\unitlength}{4.5mm}
\begin{picture}(7,1)(0,4)
\put(0,3){\line(0,1){1}}
\put(1,2){\line(0,1){1}}
\put(2,1){\line(0,1){1}}
\put(3,3){\line(0,1){1}}
\put(5,2){\line(0,1){1}}
\put(6,1){\line(0,1){1}}
\put(0,3){\line(1,0){1}}
\put(0,4){\line(1,0){3}}
\put(1,2){\line(1,0){1}}
\put(2,1){\line(1,0){4}}
\put(3,3){\line(1,0){2}}
\put(5,2){\line(1,0){1}}
\put(0.35,3.25){2}
\put(0.35,4.25){5}
\put(2.35,3.25){6}
\put(1.35,2.25){1}
\put(2.35,2.25){3}
\put(3.35,2.25){3}
\put(1.35,3.25){4}
\put(2.35,1.25){1}
\put(3.35,1.25){1}
\put(4.35,2.25){3}
\put(5.35,2.25){7}
\put(4.35,1.25){2}
\put(5.35,1.25){3}
\put(6.35,1.25){3}
\end{picture}
\quad
\begin{picture}(9,1)(0,4)
\put(0,3){\line(0,1){1}}
\put(1,2){\line(0,1){1}}
\put(2,1){\line(0,1){1}}
\put(3,3){\line(0,1){1}}
\put(5,2){\line(0,1){1}}
\put(6,1){\line(0,1){1}}
\put(0,3){\line(1,0){1}}
\put(0,4){\line(1,0){3}}
\put(1,2){\line(1,0){1}}
\put(2,1){\line(1,0){4}}
\put(3,3){\line(1,0){2}}
\put(5,2){\line(1,0){1}}
\put(0.35,3.25){2}
\put(1.35,3.25){5}
\put(2.35,3.25){6}
\put(1.35,2.25){1}
\put(2.35,2.25){3}
\put(3.35,2.25){3}
\put(4.35,2.25){4}
\put(2.35,1.25){1}
\put(3.35,1.25){1}
\put(4.35,1.25){3}
\put(5.35,1.25){7}
\put(6,0){\dashbox{0.25}(3,1){}}
\put(6.35,0.25){2}
\put(7.35,0.25){3}
\put(8.35,0.25){3}
\end{picture}
\quad
\begin{picture}(7,1)(0,4)
\put(0,3){\line(0,1){1}}
\put(1,2){\line(0,1){1}}
\put(2,1){\line(0,1){1}}
\put(3,3){\line(0,1){1}}
\put(5,2){\line(0,1){1}}
\put(6,1){\line(0,1){1}}
\put(0,3){\line(1,0){1}}
\put(0,4){\line(1,0){3}}
\put(1,2){\line(1,0){1}}
\put(2,1){\line(1,0){4}}
\put(3,3){\line(1,0){2}}
\put(5,2){\line(1,0){1}}
\put(0.35,3.25){2}
\put(0.35,4.25){5}
\put(2.35,3.25){6}
\put(1.35,2.25){1}
\put(2.35,2.25){3}
\put(3.35,2.25){3}
\put(1.35,3.25){4}
\put(2.35,1.25){1}
\put(3.35,1.25){1}
\put(4.35,2.25){3}
\put(5.35,2.25){7}
\put(4.35,1.25){2}
\put(5.35,1.25){3}
\put(6.35,1.25){3}
\end{picture}
\]
\vspace{3\unitlength}
\end{example}

The final reverse bumping path in a downward slide will play an
important role in the sign-reversing involution. Therefore, with
notation as in Definition~\ref{def:downward}, if $m < |\lp/\la|$, then
we refer to the reverse bumping path of $k_m$ as the \emph{downward
  path of $T$}. The cells of the downward path of $T$ are circled
above. Say that the downward path of $T$ \emph{\isgood} if its bottom
cell (which may be empty) is strictly below the bottom cell $\mu/\mm$.
Our terminology is justified since one can show that the \isgood\
condition is equivalent to the bottom cell of the downward path being
weakly right of the bottom cell of $\mu/\mm$.  The importance of the
\isgood\ condition is revealed by the following result.

\begin{proposition}\label{prop:downshape}
  Suppose $T$ is an \SSYT\ of shape $\lp/\mm$ such that
  the downward path of $T$, if it exists, \isgood. Then $\D(T)$ is
  an \SSYT\ of shape $\la'/\mu'$, where $\la'/\la$
  (resp.\ $\mu/\mu'$) is a horizontal (resp.\ vertical) strip.
\end{proposition}

\begin{proof}
  First observe that if $T$ has no downward path, then $\D(T) = T$ and
  clearly has the required shape. Otherwise, using the notation from
  Definition~\ref{def:downward}, for each $1 \leq i < m$, $c_i$ has a
  reverse bumping path leading down to the first row. Therefore by
  Lemma~\ref{lem:horizontal}(c), the reverse bumping path of $c_i$
  lies strictly right of that of and $c_{i+1}$, and $k_{i+1} \leq k_i$
  for $1 \leq i < m-1$. After the last reverse row insertion, which is
  along the downward path, a cell containing $k_m$ is added to the
  left end of a row strictly below the bottom cell of $\mu/\mm$.  Thus
  $\mu/\mu'$ is indeed a vertical strip. By (a) and (b) of
  Lemma~\ref{lem:horizontal}, since $k_{m-1} \leq \cdots \leq k_1$,
  reinserting these entries adds a horizontal strip, so it remains to
  show that the cell added when inserting $k_{m-1}$ lies strictly
  right of those cells of $\lp/\la$ that were not moved under the
  downward slide.  This will be the case if the cell added when
  inserting $k_{m-1}$ lies weakly right of $c_{m}$.  Indeed, the
  bumping path of $k_{m-1}$ will either be the reverse bumping path of
  $c_{m-1}$ or it will have been affected by the changes on the
  downward path and so will intersect the downward path.  In the
  former case the bumping path will end with the addition of
  $c_{m-1}$, while in the latter it ends with the addition of $c_m$,
  as required.
\end{proof}

Supposing $\D(S) = T$ with $T \neq S$, the next step is to invert the
downward slides for such $T$.  Since any such $T$ necessarily has $\mm
\neq \mu$, the idea is to internally row insert the bottom cell of
$\mu/\mm$.  However, before doing this we must reverse row insert
certain cells of $\lp/\la$, as in a downward slide. To describe which
cells to reverse insert, we define the \emph{upward path of $T$} to be
the bumping path that would result from internal row insertion of the
entry in the bottom cell of $\mu/\mm$.  Roughly, we will reverse row
insert anything that is weakly right of this upward path.

\begin{definition}\label{def:upward}
  Let $T$ be an \SSYT\ of shape $\lp/\mm$ such that $\mm \neq \mu$.
  Define the \emph{upward slide of $T$}, denoted $\U(T)$, as follows:
  construct $T \rightarrow c_1$ where $c_1$ is the rightmost cell of
  $\lp/\la$, and let $k_1$ denote its final entry and $B_1$ its
  bumping path.  If $B_1$ fails to stay weakly right of the upward
  path of $T$, then set $m=m'=0$.  Otherwise, consider $c_2$, the second
  rightmost cell of $\lp/\la$, and $k_2$, the final entry of $(T
  \rightarrow c_1) \rightarrow c_2$, and $B_2$, the corresponding
  bumping path. Continue until the last time $B_{m}$ stays weakly
  right of the upward path of $T$ or until no cell of $\lp/\la$
  remains.  Suppose that after the reverse row insertions, the bottom cell of
  $\mu/\mm$ is in row $r$ and has entry $k$.  
  Then $\U(T)$ is given by
  \begin{equation}\label{equ:upward}
  (\cdots((( (\cdots(T \rightarrow c_1) \rightarrow
  c_2\cdots) \rightarrow c_m) \intern_r k) \rowinsert
  k_{m'}) \cdots) \rowinsert k_{1} 
  \end{equation} 
where we set $m'=m$ if $k_m$ lands in row $0$, and $m'=m-1$ otherwise. 
\end{definition}

\begin{example}\label{exa:upward}
  Letting $T$ be the rightmost \SSYT\ of~\eqref{equ:downward}, the
  cells of the upward path of $T$ are circled below.  We determine
  $\U(T)$ in three steps.  We find that $m=3$ and the middle \SSYT\
  of~\eqref{equ:downward} shows $((T \rightarrow c_1) \rightarrow c_2)
  \rightarrow c_3$.  Then $(((T \rightarrow c_1) \rightarrow c_2)
  \rightarrow c_3) \intern k$ is shown in the middle below, while
  $\U(T)$ is shown on the right.  Comparing with
  Example~\ref{exa:downward}, we observe that the upward slide in this
  case does indeed invert the downward slide.
\[
\setlength{\unitlength}{4.5mm}
\begin{picture}(7,1)(0,6)
\put(0,3){\line(0,1){2}}
\put(1,4){\line(0,1){1}}
\put(2,2){\line(0,1){1}}
\put(3,1){\line(0,1){1}}
\put(4,3){\line(0,1){1}}
\put(5,2){\line(0,1){1}}
\put(7,1){\line(0,1){1}}
\put(0,5){\line(1,0){1}}
\put(0,3){\line(1,0){2}}
\put(1,4){\line(1,0){3}}
\put(2,2){\line(1,0){1}}
\put(3,1){\line(1,0){4}}
\put(4,3){\line(1,0){1}}
\put(5,2){\line(1,0){2}}
\put(0.35,6.25){\hcirc{\ }}
\put(0.35,5.25){\hcirc{9}}
\put(0.35,4.25){3}
\put(1.35,4.25){\hcirc{5}}
\put(2.35,4.25){7}
\put(0.35,3.25){2}
\put(1.35,3.25){\hcirc{3}}
\put(2.35,3.25){4}
\put(3.35,3.25){7}
\put(1.35,2.25){\hcirc{2}}
\put(2.35,2.25){2}
\put(3.35,2.25){2}
\put(4.35,2.25){3}
\put(5.35,2.25){6}
\put(2.35,1.25){\hcirc{1}}
\put(3.35,1.25){1}
\put(4.35,1.25){2}
\put(5.35,1.25){2}
\put(6.35,1.25){5}
\end{picture}
\quad
\begin{picture}(10,1)(0,6)
\put(0,3){\line(0,1){2}}
\put(1,4){\line(0,1){1}}
\put(2,2){\line(0,1){1}}
\put(3,1){\line(0,1){1}}
\put(4,3){\line(0,1){1}}
\put(5,2){\line(0,1){1}}
\put(7,1){\line(0,1){1}}
\put(0,5){\line(1,0){1}}
\put(0,3){\line(1,0){2}}
\put(1,4){\line(1,0){3}}
\put(2,2){\line(1,0){1}}
\put(3,1){\line(1,0){4}}
\put(4,3){\line(1,0){1}}
\put(5,2){\line(1,0){2}}
\put(0.35,5.25){9}
\put(0.35,4.25){3}
\put(1.35,4.25){5}
\put(0.35,3.25){2}
\put(1.35,3.25){2}
\put(2.35,3.25){7}
\put(3.35,3.25){7}
\put(1.35,2.25){1}
\put(2.35,2.25){2}
\put(3.35,2.25){3}
\put(4.35,2.25){4}
\put(2.35,1.25){\ }
\put(3.35,1.25){2}
\put(4.35,1.25){2}
\put(5.35,1.25){3}
\put(6.35,1.25){6}
\put(7,0){\dashbox{0.25}(3,1){}}
\put(7.35,0.25){1}
\put(8.35,0.25){2}
\put(9.35,0.25){5}
\end{picture}
\quad\!
\begin{picture}(7,1)(0,6)
\put(0,3){\line(0,1){2}}
\put(1,4){\line(0,1){1}}
\put(2,2){\line(0,1){1}}
\put(3,1){\line(0,1){1}}
\put(4,3){\line(0,1){1}}
\put(5,2){\line(0,1){1}}
\put(7,1){\line(0,1){1}}
\put(0,5){\line(1,0){1}}
\put(0,3){\line(1,0){2}}
\put(1,4){\line(1,0){3}}
\put(2,2){\line(1,0){1}}
\put(3,1){\line(1,0){4}}
\put(4,3){\line(1,0){1}}
\put(5,2){\line(1,0){2}}
\put(0.35,5.25){9}
\put(0.35,4.25){3}
\put(1.35,4.25){5}
\put(2.35,4.25){7}
\put(3.35,4.25){7}
\put(0.35,3.25){2}
\put(1.35,3.25){2}
\put(2.35,3.25){3}
\put(3.35,3.25){4}
\put(1.35,2.25){1}
\put(2.35,2.25){2}
\put(3.35,2.25){2}
\put(4.35,2.25){3}
\put(5.35,2.25){6}
\put(3.35,1.25){1}
\put(4.35,1.25){2}
\put(5.35,1.25){2}
\put(6.35,1.25){5}
\end{picture}
\vspace{6\unitlength}
\]
There are also instances where the entry $k$ of
Definition~\ref{def:upward} is different from the entry originally at
the bottom of the upward path.  For example, the same three-step
process for constructing $\U(T)$ is shown below for an example with
$m=2$.  There, $k=2$, even though the original upward path of $T$ had
1 as its bottom entry.
\[
 \setlength{\unitlength}{4.5mm}
\begin{picture}(2,2)(0,2)
\put(0,2){\line(0,1){1}}
\put(1,1){\line(0,1){1}}
\put(2,1){\line(0,1){2}}
\put(0,3){\line(1,0){2}}
\put(0,2){\line(1,0){1}}
\put(1,1){\line(1,0){1}}
\put(0.35,3.25){\hcirc{3}}
\put(1.35,3.25){3}
\put(0.35,2.25){\hcirc{2}}
\put(1.35,2.25){2}
\put(0.35,1.25){\hcirc{1}}
\put(1.35,1.25){1}
\end{picture}
\quad \longrightarrow \quad
\begin{picture}(4,2)(0,2)
\put(0,2){\line(0,1){1}}
\put(1,1){\line(0,1){1}}
\put(2,1){\line(0,1){2}}
\put(0,3){\line(1,0){2}}
\put(0,2){\line(1,0){1}}
\put(1,1){\line(1,0){1}}
\put(0.35,2.25){3}
\put(1.35,2.25){3}
\put(0.35,1.25){2}
\put(1.35,1.25){2}
\put(2,0){\dashbox{0.25}(2,1)}
\put(2.35,0.25){1}
\put(3.35,0.25){1}
\end{picture}
\quad \longrightarrow \quad
\begin{picture}(4,2)(0,2)
\put(0,2){\line(0,1){1}}
\put(1,1){\line(0,1){1}}
\put(2,1){\line(0,1){2}}
\put(0,3){\line(1,0){2}}
\put(0,2){\line(1,0){1}}
\put(1,1){\line(1,0){1}}
\put(0.35,3.25){3}
\put(1.35,2.25){3}
\put(0.35,2.25){2}
\put(1.35,1.25){2}
\put(2,0){\dashbox{0.25}(2,1)}
\put(2.35,0.25){1}
\put(3.35,0.25){1}
\end{picture}
\quad \longrightarrow \quad
\begin{picture}(3,2)(0,2)
\put(0,2){\line(0,1){1}}
\put(1,1){\line(0,1){1}}
\put(2,1){\line(0,1){2}}
\put(0,3){\line(1,0){2}}
\put(0,2){\line(1,0){1}}
\put(1,1){\line(1,0){1}}
\put(0.35,3.25){3}
\put(1.35,3.25){3}
\put(0.35,2.25){2}
\put(1.35,2.25){2}
\put(1.35,1.25){1}
\put(2.35,1.25){1}
\end{picture}
  \]
\vspace{\unitlength}
\end{example}

As with downward slides and before presenting our involution, we must
ensure that the result of an upward slide is always a tableau of the
appropriate skew shape.

\begin{proposition}\label{prop:upshape}
  Suppose $T$ is an \SSYT\ of shape $\lp/\mm$ such that in the upward
  slide of $T$, all the final entries of the reverse row insertions
  land in row 0. Then $\U(T)$ is an \SSYT\ of shape $\la'/\mu'$,
  where $\la'/\la$ (resp.\ $\mu/\mu'$) is a horizontal (resp.\
  vertical) strip.
\end{proposition}

\begin{proof}
  If none of the reverse bumping paths $B_1, \ldots, B_m$ in the
  upward slide of $T$ intersect the upward path, then inserting the
  bottom cell of the upward path into the row above will follow the
  upward path. So applying $\U$ to $T$ will simply amount to row
  insertion along the upward path, with any movement along the $B_i$
  being inverted in the course of applying $\U$.  Thus, with notation
  as in Definition~\ref{def:upward}, we can assume that $B_m$
  intersects the upward path of $T$, where $B_m$ is the last reverse
  bumping path that stays weakly right of the upward path of $T$.
  Note that by Lemma~\ref{lem:horizontal}(c), no other reverse bumping
  path $B_{m'}$ with $m'<m$ can intersect the upward path.  When $k$
  is internally inserted into $T$, it will follow the upward path of
  $T$ until it intersects $B_m$ from which time it will follow $B_m$,
  ultimately adding $c_m$ back to the tableau. Next inserting $k_m$
  will result in a bumping path $B_m'$ following $B_m$ until one row
  below the point of intersection.  Note that $B_m'$ cannot intersect
  the bottom cell of the upward path, since that cell is now empty.
  Therefore, in the row of the bottom cell of the upward path, $B_m'$
  is strictly right of the upward path.  Thus we can apply
  Lemma~\ref{lem:horizontal}(a) to deduce that $B_m'$ will necessarily
  remain strictly right of the bumping path for $k$.  Finally, by
  Lemma~\ref{lem:horizontal}(c), we have that $k_m \leq \cdots \leq
  k_1$. Thus, by (a) and (b) Lemma~\ref{lem:horizontal}, inserting the
  remaining entries still results in the addition of a horizontal
  strip.
\end{proof}

Our involution will consist of either applying a downward slide or an
upward slide. The decision for which slide to apply is roughly based
on which of the downward path of $T$ or the upward path of $T$ lies
further to the right.

\begin{definition}\label{def:involution}
  Consider the set of \SSYTs\ $T$ of shape $\lp/\mm$ such that that
  $\lp/\la$ is a horizontal strip and $\mu/\mm$ is a vertical strip.
  Define a map $\phi$ on such $T$ by
  \[  \phi(T) \ = \ \left\{ \begin{array}{rl}
      \D(T) &        \mbox{if $T$ has no upward path or }\\
      & \mbox{the downward path of $T$ exists and \isgood,} \\[1ex]
      \U(T) & \mbox{otherwise}.
    \end{array} \right.
  \]
\end{definition}

\begin{theorem}\label{thm:involution}
  The map $\phi$ defines an involution on the set of \SSYTs\ with
  shapes of the form $\lp/\mm$ where $\lp/\la$ is a horizontal strip
  and $\mu/\mm$ is a vertical strip.
\end{theorem}

\begin{proof}
  We refer the reader to Examples~\ref{exa:downward}
  and~\ref{exa:upward} for illustrations of the ideas of the proof.

  First suppose that $\phi(T) = \D(T)$. By
  Definition~\ref{def:involution}, the first way in which this can
  happen is if $T$ has neither an upward path nor a downward path.
  This corresponds to the case $m'=m$ in Definition~\ref{def:downward}
  which, by Lemma~\ref{lem:inverse}(b), results in $\phi(T)=\D(T)=T$,
  trivially an involution.
  
  Therefore, assume the downward path of $T$ exists and \isgood, so by
  Proposition~\ref{prop:downshape}, $\phi(T)$ has the required
  shape. Moreover, since $\D(T)$ adds a cell to $\mu/\mm$, $\phi(T)$
  must have an upward path.  Suppose cell $c_m$ is at the top of the
  downward path of $T$. It follows from Lemma~\ref{lem:horizontal}
  that all cells strictly right of $c_m$ and above $\la$ in $T$ have
  reverse bumping paths strictly right of the downward path of $T$ and,
  by definition of the downward path, their reverse row insertions all land in row 0.
  For any cell $c$ strictly left of $c_m$ in $T$,
  Lemma~\ref{lem:horizontal}(c) implies that the reverse bumping path
  of $c$ in $T$ must remain strictly left of the downward path of $T$,
  and this property persists in $\phi(T)$. Therefore either $\phi(T)$
  has no downward path or the downward path does not \begood, and so
  $\phi(\phi(T)) = \U(\D(T))$.
  
  To show that $\U(\D(T)) = T$, first consider $\D$ applied to $T$.
  After reverse row inserting along the reverse bumping paths
  including the downward path, suppose we have arrived at an SSYT $S$.
  We next row insert the final entries $k_{m-1}, \ldots, k_1$ of
  Definition~\ref{def:downward}.  These new bumping paths may have
  been affected by the changes on the downward path but, as shown in
  the proof of Proposition~\ref{prop:downshape}, the bumping path $B$
  of $k_{m-1}$ still lies weakly right of the downward path.  Next, we
  consider the application of $\U$ to $\D(T)$ and observe that, since
  the downward path of $T$ exits right, it shares a bottom cell with
  the upward path of $\D(T)$.  Thus the upward path must stay weakly
  left of the downward path and, in particular, $B$ lies weakly right
  of the upward path.  Moreover, $B$ corresponds to the $B_m$ of
  Definition~\ref{def:upward} since, as mentioned in the previous
  paragraph, reverse bumping paths further left will have a bottom
  cell that is strictly left of the bottom cell of the upward path.  A
  key idea is now evident: after we perform the first part of $\U$ by
  applying the reverse row insertions, we will return to the SSYT $S$.
  Therefore, the rest of the application of $\U$ to $\D(T)$ will
  invert the reverse row insertions from the application of $\D$ to
  $T$, as required.

  Next, suppose $\phi(T) = \U(T)$. By Definition~\ref{def:involution},
  $T$ has an upward path and either has no downward path or the
  downward path does not \begood.  So, the downward path, if it
  exists, does not stay weakly right of the upward path.  In
  particular, all the final entries of the reverse row insertions in
  the upward slide land in row 0.  So by
  Proposition~\ref{prop:upshape}, $\phi(T)$ has the required shape.
 
  If none of the reverse bumping paths $B_1, \ldots, B_m$ in the
  upward slide of $T$ intersect the upward path then, as we observed
  in the proof of Proposition~\ref{prop:upshape}, applying $\U$ to $T$
  will simply amount to row insertion along the upward path, with any
  movement along the $B_i$ being inverted in the course of applying
  $\U$.  In particular, the upward path of $T$ will become the
  downward path of $\U(T)$ and hence it \isgood.  Again, the downward
  path of $\U(T)$ will not intersect the other reverse bumping paths
  of $\U(T)$, which will still be $B_1, \ldots, B_m$.  So applying
  $\D$ to $\U(T)$ will simply amount to the reverse row insertion
  along the downward path.  Thus $\phi(\phi(T)) = \D(\U(T)) = T$.
  
  On the other hand, if $B_m$ from Definition~\ref{def:upward}
  intersects the upward path, we know from the proof of
  Proposition~\ref{prop:upshape} that internally inserting the bottom
  entry $k$ of the upward path into the row above will eventually
  follow $B_m$ and all other insertions will have bumping paths $B_m',
  B_{m-1}', \ldots, B_1'$ strictly to the right. Therefore this
  bumping path for $k$ will become the downward path of $\U(T)$ and
  again it clearly \isgood.  Thus when we apply $\D$ to $\U(T)$, we
  will first reverse row insert along the $B_i'$ and then invert the
  changes caused by the internal insertion of $k$.  As a result, when
  we complete the $\D$ operation by reinserting the final entries of
  $B_m', B_{m-1}', \ldots, B_1'$, the bumping paths will be the same
  as the original reverse bumping paths when we applied $\U$ to $T$.
  Thus $\phi(\phi(T)) = \D(\U(T)) = T$.
\end{proof}

We now have all the ingredients needed to prove the skew Pieri rule.

\begin{proof}[Proof of Theorem~\ref{thm:skewpieri}]
  Using the expansion of $s_{\lp/\mm}$ in terms of \SSYTs\ as in
  \eqref{equ:schur}, observe that if $\phi(T) \neq T$, then the $T$
  and $\phi(T)$ occur with different signs in the right-hand side of
  \eqref{equ:skewpieri}. Since $\phi$ clearly preserves the monomial
  associated to an SSYT, the monomials corresponding to $T$ and
  $\phi(T)$ in the right-hand side of \eqref{equ:skewpieri} will
  cancel out.  Because $s_{\la/\mu} s_n = s_{(\la/\mu) * (n)}$, it
  remains to show that there is a monomial-preserving bijection from
  fixed points of $\phi$ to \SSYTs\ of shape $(\la/\mu) * (n)$.

  Note that $T$ is a fixed point of $\phi$ only if $T$ has neither an
  upward path nor a downward path. This happens if and only if
  $\mm=\mu$ and when reverse row inserting the cells of $\lp/\la$ from
  right to left, every final entry lands in row 0.  In particular, the
  entries of $T$ remaining after reverse row inserting the cells of
  $\lp/\la$ form an SSYT of shape $\la/\mu$.  Say the final entries of
  the reverse row insertions are $k_n, \ldots, k_1$ in the order
  removed. By Lemma~\ref{lem:horizontal}(c), since $\lp/\la$ is a
  horizontal strip, we have $k_1 \leq \cdots \leq k_n$ and so these
  entries form an SSYT of shape $(n)$.  By Lemma~\ref{lem:inverse},
  this process is invertible and therefore establishes the desired
  bijection.
\end{proof}

\section{Concluding remarks}\label{sec:conclusion}

\subsection{Littlewood--Richardson fillings}

We proved Theorem~\ref{thm:skewpieri} by working with SSYTs.  In
particular, we showed that the two sides of~\eqref{equ:skewpieri} were
equal when expanded in terms of monomials.  Alternatively, we could
consider the expansions of both sides of~\eqref{equ:skewpieri} in
terms of Schur functions.  The Littlewood--Richardson rule states that
the coefficient of $s_\nu$ in the expansion of any skew Schur function
$s_{\lambda/\mu}$ is the number of \emph{Littlewood--Richardson
  fillings} (\emph{LR-fillings}) of shape $\lambda/\mu$ and content
$\nu$.  (The interested reader unfamiliar with LR-fillings can find
the definition in \cite{ECII}.)  It is not hard to check that our maps
$\D$ and $\U$ send LR-fillings to LR-fillings, and bumping within
LR-fillings has some nice properties. For example, the entries along a
(reverse) bumping path are always $1,2,\ldots,r$ from bottom to top
for some $r$.

However, we chose to give our proof in terms of SSYTs because one does
not need to invoke the power of the Littlewood--Richardson rule to
prove the classical Pieri rule, and we wanted the same to apply to the
skew Pieri rule.

\subsection{The bigger picture}\label{sub:bigpicture}

We would like to conclude by asking if the skew Pieri rule can be
shown to be a special case of a larger framework.  For example,
\cite{Fom95} and \cite{Lam06} both give general setups that might be
relevant. It would be of obvious interest if these frameworks or any
others in the literature could be used to rederive the skew Pieri
rule.

In a similar spirit, we close with a conjectural rule for the product
$s_{\la/\mu} s_{\si/\tau}$, which has now been proved by Lam, Lauve and Sottile in \cite{LLS09pr}.  This rule gives the skew Pieri rule when
$\si/\tau = (n)$ and gives the classical Littlewood--Richardson rule
when $\mu=\tau=\emptyset$.  The Littlewood--Richardson rule can itself
be used to evaluate the general product $s_{\la/\mu} s_{\si/\tau}$,
but our rule will be different in that, like the skew Pieri rule, we
will be adding boxes to both the inside and outside of $\la/\mu$.

As before, fix a skew shape $\la/\mu$, and suppose we have a skew
shape $\lp/\mm$ such that $\lp \supseteq \la$ and $\mm \subseteq \mu$.
We no longer require that $\lp/\la$ (resp.\ $\mu/\mm$) is a horizontal
(resp.\ vertical) strip.  We will need a few new definitions that are
variants of those that arise in the Littlewood--Richardson rule.  Let
$T^+$ be an SSYT of shape $\lp/\la$ and let $T^-$ be a filling of
$\mu/\mm$.  We let $T^-$ be an \emph{\assyt} (\ASSYT), meaning that
the entries of $T^-$ strictly decrease from left-to-right along rows,
and weakly decrease up columns.  When $\la/\mu=7541/33$ and
$\lp/\mm=9953/1$, an example of a pair $(T^-,T^+)$ with the
above-stated properties is
\begin{equation}\label{equ:skewbasedlr}
\setlength{\unitlength}{4.5mm}
\begin{picture}(9,2)(0,2)
\put(0,2){\line(0,1){2}}
\put(1,3){\line(0,1){1}}
\put(3,0){\line(0,1){2}}
\put(4,2){\line(0,1){1}}
\put(5,1){\line(0,1){1}}
\put(7,0){\line(0,1){1}}
\put(0,4){\line(1,0){1}}
\put(0,2){\line(1,0){3}}
\put(1,3){\line(1,0){3}}
\put(3,0){\line(1,0){4}}
\put(4,2){\line(1,0){1}}
\put(5,1){\line(1,0){2}}
\put(1.35,3.25){5}
\put(2.35,3.25){6}
\put(4.35,2.25){3}
\put(7.35,0.25){2}
\put(8.35,0.25){4}
\put(5.35,1.25){1}
\put(6.35,1.25){4}
\put(7.35,1.25){4}
\put(8.35,1.25){5}
\put(0.35,1.25){5}
\put(1.35,0.25){3}
\put(1.35,1.25){3}
\put(2.35,1.25){1}
\put(2.35,0.25){2}
\end{picture}\ .
\end{equation}
\vspace{1.5\unitlength}

The reverse reading word of the pair $(T^-, T^+)$ is the sequence of
entries obtained by first reading the entries of $T^-$ from
bottom-to-top along its columns, starting with rightmost column and
moving left, followed by reading the entries of $T^+$ from
right-to-left along its rows, starting with the bottom row and moving
upwards.  The reverse reading word of the example above is
21335425441365.  A word $w$ is said to be Yamanouchi if, in the first
$j$ letters of $w$, the number of occurrences of $i$ is no less than
the number of occurrences of $i+1$, for all $i$ and $j$.
For a partition $\tau$, a word $w$ is said to be $\tau$-Yamanouchi if
it is Yamanouchi when prefixed by the following concatenation:
$\tau_1$ 1's, followed by $\tau_2$ 2's, and so on.  The reverse
reading word of~\eqref{equ:skewbasedlr} is certainly not Yamanouchi
but is 5321-Yamanouchi.

For a proof of the following conjecture, see Theorem~6 and Remark~7 in \cite{LLS09pr}.
\begin{conjecture}\label{con:skewbasedlr}  For any skew shapes $\la/\mu$ and $\si/\tau$, 
\[
s_{\la/\mu} s_{\si/\tau} = \sum_{\substack{T^- \in
    \mathrm{\ASSYT}(\mu/\mm) \\ T^+ \in \mathrm{\SSYT}(\lp/\la)}}
(-1)^{|\mu/\mm|} s_{\lp/\mm},
\]
where the sum is over all \ASSYTs\ $T^-$ of shape $\mu/\mm$ for some
$\mm \subseteq \mu$, and \SSYTs\ $T^+$ of shape $\lp/\la$ for some
$\lp \supseteq \la$, with the following properties:
\begin{itemize}
\item the combined content of $T^-$ and $T^+$ is the component-wise
  difference $\si-\tau$, and
\item the reverse reading word of $(T^-,T^+)$ is $\tau$-Yamanouchi.
\end{itemize}
\end{conjecture}
For example, to the product $s_{7541/33}\, s_{755431/5321}$, the pair
$(T^-,T^+)$ of~\eqref{equ:skewbasedlr} contributes $-s_{9953/1}$.
Note that when $\mu=\tau=\emptyset$, Conjecture~\ref{con:skewbasedlr}
is exactly the classical Littlewood--Richardson rule.  If instead
$\si/\tau = (n)$, then all the entries of $T^-$ and $T^+$ must be 1's,
and the skew Pieri rule results.

Another interesting special case of Conjecture~\ref{con:skewbasedlr} is
when $\si/\tau$ is a horizontal strip with row lengths
$\rho=\sigma-\tau$ from bottom to top, with $\rho$ a partition.  In
this case, $s_{\si/\tau} = h_\rho$, while the $\tau$-Yamanouchi
property is trivially satisfied for any $(T^-,T^+)$ with content
$\si-\tau$.  Conjecture~\ref{con:skewbasedlr} then gives the following
expression for the product $s_{\la/\mu} h_\rho$ for any skew shape
$\la/\mu$ and partition $\rho$:
\[
s_{\la/\mu} h_\rho = \sum_{\substack{T^- \in \mathrm{\ASSYT}(\mu/\mm) \\ T^+ \in \mathrm{\SSYT}(\lp/\la)}} (-1)^{|\mu/\mm|} s_{\lp/\mm},
\]
where the sum is over all \ASSYTs\ $T^-$ of shape $\mu/\mm$ for some
$\mm \subseteq \mu$, and \SSYTs\ $T^+$ of shape $\lp/\la$ for some
$\lp \supseteq \la$, such that the combined content of $T^-$ and $T^+$
is $\rho$.  Observe that this result follows directly from repeated
applications of the skew Pieri rule.


\section{Acknowledgments} 
We are grateful to a number of experts for informing us that they too
were surprised by the existence of the skew Pieri rule, and
particularly to Richard Stanley for providing an algebraic proof of
the $n=1$ case that preceded our combinatorial proof. We are also
grateful to Thomas Lam for extending Stanley's proof to the general
case and for his willingness to append his proof to this manuscript.
This research was performed while the second author was visiting MIT,
and he thanks the mathematics department for their hospitality.
Conjecture verification was performed using
\cite{BucSoftware,SteSoftware}.

\appendix
\section{An algebraic proof (by Thomas Lam)}\label{sec:algebraic}
Let $\langle .,. \rangle: \Lambda \times \Lambda \to \mathbb{Q}$ denote the (symmetric, bilinear) Hall inner product.
For $f \in \Lambda$, we let $f^\perp: \Lambda \to \Lambda$ denote the adjoint operator to multiplication by $f$, so that $\langle fg, h\rangle = \langle g,f^\perp h \rangle$ for $g,h \in \Lambda$.  

\begin{proposition}\label{prop:appendix}
For $n \geq 1$, and $f, g \in \Lambda$, we have
\begin{equation}\label{E:prop}
f \, h_n^\perp(g) = \sum_{k=0}^n (-1)^k h_{n-k}^\perp(e_k^\perp(f)
g).
\end{equation}
\end{proposition}

\begin{proof}
We shall use the formula \cite[2.6$'$]{Mac95}, for $n \geq 1$
\begin{equation}\label{E:eh}
\sum_{i=0}^n (-1)^i e_i h_{n-i} = 0
\end{equation}
and \cite[Example I.5.25(d)]{Mac95}
\begin{equation}\label{E:h}
h_n^\perp(fg) = \sum_{i=0}^n h_{n-i}^\perp(f) h_i^\perp(g).
\end{equation}
In the following we shall use the fact that the map $\Lambda \to {\rm End}_{\mathbb{Q}}(\Lambda)$ given by $f \to f^\perp$ is a ring homomorphism \cite[Example I.5.3]{Mac95}.  Starting from the right-hand side of~\eqref{E:prop} and using \eqref{E:h} we have
\[
\sum_{k=0}^n (-1)^k h_{n-k}^\perp(e_k^\perp(f) g)
= \sum_{k=0}^n (-1)^k \sum_{j=0}^{n-k} h_j^\perp(e_k^\perp(f))
h_{n-k-j}^\perp(g).
\]
Under the substitution $n-k=j+i$, the right-hand side becomes
\begin{equation}\label{E:for_reducing}
\sum_{i=0}^n (-1)^{n-i} \left(\sum_{j=0}^{n-i} (-1)^{j}
(h_j^\perp e_{n-j-i}^\perp)(f)\right) h_i^\perp(g).
\end{equation}
Since $\sum_{j=0}^{n-i} (-1)^{j} (h_j^\perp e_{n-j-i}^\perp)(f) = 0$
for $n - i > 0$ using (\ref{E:eh})$^\perp$, but is equal to $f$ for
$n = i$, we see that~\eqref{E:for_reducing} reduces to $f\,h_n^\perp(g)$.
\end{proof}

\begin{proof}[Proof of Theorem \ref{thm:skewpieri}]
It is well known (see \cite[I.(4.8), I.(5.1)]{Mac95} or \cite[Corollary 7.12.2, (7.60)]{ECII}) that for two partitions $\la$ and $\mu$, we have $\langle s_\la,s_\mu \rangle = \delta_{\la,\mu}$ and $s_\mu^\perp s_\la = s_{\la/\mu}$, where $s_{\la/\mu} = 0$ if $\mu \not \subseteq \la$.  Let $g \in \Lambda$.  
We calculate 
\begin{align*}
\langle s_{\la/\mu} h_n, g \rangle 
= \langle s_{\la/\mu}, h_n^\perp g \rangle 
= \langle s_\mu^\perp s_\la, h_n^\perp g \rangle 
&= \langle s_\la, s_\mu h_n^\perp g \rangle \\
&= \langle s_\la,  \sum_{k=0}^n (-1)^k h_{n-k}^\perp(e_k^\perp(s_\mu)g)\rangle 
\end{align*}
by Proposition \ref{prop:appendix}.
Using the Pieri rule (Theorem \ref{thm:pieri}), this amounts to 
\[
 \left \langle\sum_{k=0}^n (-1)^k
    \sum_{\substack{\lp/\la \ (n-k)\text{-}\mathrm{hor.\ strip} \\      \mu/\mm \ k \text{-} \mathrm{vert.\ strip}}} s_{\lp/\mm}\ ,g \right\rangle.
\]
Since the Hall inner product is non-degenerate, we obtain (\ref{equ:skewpieri}).
\end{proof}



\begin{thebibliography}{10}

\bibitem{BucSoftware}
Anders~S. Buch.
\newblock {L}ittlewood--{R}ichardson calculator, 1999.
\newblock \newline Available from
  \href{http://www.math.rutgers.edu/~asbuch/lrcalc/}
  {\url{http://www.math.rutgers.edu/~asbuch/lrcalc/}}.

\bibitem{Fom95}
Sergey Fomin.
\newblock Schur operators and {K}nuth correspondences.
\newblock {\em J. Combin. Theory Ser. A}, 72(2):277--292, 1995.

\bibitem{Knu70}
Donald~E. Knuth.
\newblock Permutations, matrices, and generalized {Y}oung tableaux.
\newblock {\em Pacific J. Math.}, 34:709--727, 1970.

\bibitem{Koo88}
Tom~H. Koornwinder.
\newblock Self-duality for $q$-ultraspherical polynomials associated with root
  system $a_n$.
\newblock Unpublished manuscript. \newline
  \href{http://staff.science.uva.nl/~thk/art/informal/dualmacdonald.pdf}
  {\url{http://staff.science.uva.nl/~thk/art/informal/dualmacdonald.pdf}},
  1988.

\bibitem{Lam05}
Thomas Lam.
\newblock Ribbon tableaux and the {H}eisenberg algebra.
\newblock {\em Math. Z.}, 250(3):685--710, 2005.

\bibitem{Lam06}
Thomas Lam.
\newblock A combinatorial generalization of the boson-fermion correspondence.
\newblock {\em Math. Res. Lett.}, 13(2-3):377--392, 2006.

\bibitem{LLS09pr}
Thomas Lam, Aaron Lauve, and Frank Sottile.
\newblock Skew {L}ittlewood--{R}ichardson rules from {H}opf algebras.
\newblock Preprint. \url{http://arxiv.org/abs/0908.3714}, 2009.

\bibitem{LaSc82}
Alain Lascoux and Marcel-Paul Sch{\"u}tzenberger.
\newblock Polyn\^omes de {S}chubert.
\newblock {\em C. R. Acad. Sci. Paris S\'er. I Math.}, 294(13):447--450, 1982.

\bibitem{Las89}
Michel Lassalle.
\newblock Une formule de {P}ieri pour les polyn\^omes de {J}ack.
\newblock {\em C. R. Acad. Sci. Paris S\'er. I Math.}, 309(18):941--944, 1989.

\bibitem{LeSo07}
Cristian Lenart and Frank Sottile.
\newblock A {P}ieri-type formula for the {$K$}-theory of a flag manifold.
\newblock {\em Trans. Amer. Math. Soc.}, 359(5):2317--2342 (electronic), 2007.

\bibitem{LiRi34}
D.E. Littlewood and A.R. Richardson.
\newblock Group characters and algebra.
\newblock {\em Philos. Trans. Roy. Soc. London, Ser. A}, 233:99--141, 1934.

\bibitem{Mac95}
I.~G. Macdonald.
\newblock {\em Symmetric functions and {H}all polynomials}.
\newblock Oxford Mathematical Monographs. The Clarendon Press Oxford University
  Press, New York, second edition, 1995.
\newblock With contributions by A. Zelevinsky, Oxford Science Publications.

\bibitem{Man98}
Laurent Manivel.
\newblock {\em Fonctions sym\'etriques, polyn\^omes de {S}chubert et lieux de
  d\'eg\'en\'erescence}, volume~3 of {\em Cours Sp\'ecialis\'es [Specialized
  Courses]}.
\newblock Soci\'et\'e Math\'ematique de France, Paris, 1998.

\bibitem{Mor64}
A.~O. Morris.
\newblock A note on the multiplication of {H}all functions.
\newblock {\em J. London Math. Soc.}, 39:481--488, 1964.

\bibitem{Pie1893}
Mario Pieri.
\newblock Sul problema degli spazi secanti.
\newblock {\em Rend. Ist. Lombardo (2)}, 26:534--546, 1893.

\bibitem{Rob1938}
G.~de~B. Robinson.
\newblock On the {R}epresentations of the {S}ymmetric {G}roup.
\newblock {\em Amer. J. Math.}, 60(3):745--760, 1938.

\bibitem{SaSt90}
Bruce~E. Sagan and Richard~P. Stanley.
\newblock Robinson--{S}chensted algorithms for skew tableaux.
\newblock {\em J. Combin. Theory Ser. A}, 55(2):161--193, 1990.

\bibitem{Sch1961}
C.~Schensted.
\newblock Longest increasing and decreasing subsequences.
\newblock {\em Canad. J. Math.}, 13:179--191, 1961.

\bibitem{Sch77}
M.-P. Sch{\"u}tzenberger.
\newblock La correspondance de {R}obinson.
\newblock In {\em Combinatoire et repr\'esentation du groupe sym\'etrique
  (Actes Table Ronde CNRS, Univ. Louis-Pasteur Strasbourg, Strasbourg, 1976)},
  pages 59--113. Lecture Notes in Math., Vol. 579. Springer, Berlin, 1977.

\bibitem{Sot96}
Frank Sottile.
\newblock Pieri's formula for flag manifolds and {S}chubert polynomials.
\newblock {\em Ann. Inst. Fourier (Grenoble)}, 46(1):89--110, 1996.

\bibitem{Sta89}
Richard~P. Stanley.
\newblock Some combinatorial properties of {J}ack symmetric functions.
\newblock {\em Adv. Math.}, 77(1):76--115, 1989.

\bibitem{ECII}
Richard~P. Stanley.
\newblock {\em Enumerative combinatorics. {V}ol. 2}, volume~62 of {\em
  Cambridge Studies in Advanced Mathematics}.
\newblock Cambridge University Press, Cambridge, 1999.

\bibitem{SteSoftware}
John~R. Stembridge.
\newblock {SF}, posets and coxeter/weyl.
\newblock \newline Available from
  \href{http://www.math.lsa.umich.edu/~jrs/maple.html}
  {\url{http://www.math.lsa.umich.edu/~jrs/maple.html}}.

\bibitem{ThoThesis}
Gl{\^a}nffrwd~P. Thomas.
\newblock {\em {B}axter algebras and {S}chur functions}.
\newblock PhD thesis, University College of Swansea, 1974.

\bibitem{Tho78}
Gl{\^a}nffrwd~P. Thomas.
\newblock On {S}chensted's construction and the multiplication of {S}chur
  functions.
\newblock {\em Adv. in Math.}, 30(1):8--32, 1978.

\bibitem{Win98}
Rudolf Winkel.
\newblock On the multiplication of {S}chubert polynomials.
\newblock {\em Adv. in Appl. Math.}, 20(1):73--97, 1998.

\end{thebibliography}
\end{document}